\documentclass{amsart}

\usepackage{epsf}
\usepackage{amssymb}
\usepackage{amscd}

\newtheorem{definition}{Definition}
\newtheorem{example}{Example}
\newtheorem{theorem}{Theorem}
\newtheorem{proposition}{Proposition}
\newtheorem{conjecture}{Conjecture}
\newtheorem{remark}{Remark}

\begin{document}

\title{On representations of dialgebras and conformal algebras}

\author{Pavel Kolesnikov}

\address{Sobolev Institute of Mathematics \\
Novosibirsk, Russia}
\email{pavelsk@math.nsc.ru}

\begin{abstract}
In this note, we observe a relation between dialgebras
(in particular, Leibniz algebras) and
conformal algebras. The purpose is to show how
the methods of conformal algebras help solving problems on
dialgebras, and, conversely, how the ideas of dialgebras work
for conformal algebras.
\end{abstract}

\keywords{Leibniz algebra; Dialgebra; Vertex operator algebra; Conformal algebra.}

\maketitle

\section{Conformal Algebras}
The notion of a conformal algebra was introduced
in \cite{Kac1998} (in \cite{Primc1999}, a similar notion
appeared under the name of a vertex Lie algebra).
This notion is an important tool for studying vertex operator
algebras. The latter came into algebra from mathematical physics
(namely, from the 2-dimensional conformal field theory, what
explains the name ``conformal algebra"), that was initiated by
\cite{BPZ1984}. The algebraic essence of vertex operator structures
was extracted in \cite{Bor1986} and later developed in a series of
works, e.g., \cite{FLM1988, DongLep1993, Li1996}.
The relations between vertex and conformal algebras are very much
similar to the relations between ordinary associative and Lie algebras.

In conformal field theory, the
operator product expansion (OPE) describes the
commutator of two fields.
Let $V$ be a (complex) space of states, and let $Y:V\to \mathrm{End}\,V[[z,z^{-1}]]$,
$Y:b\mapsto Y(b,z)$,
be a state-field correspondence of a vertex algebra.
Then the
commutator of two fields can be expressed as a finite distribution
\[
[Y(a,w),Y(b,z)] = \sum\limits_{n\ge 0} \dfrac{1}{n!} Y(c_n, z)
 \dfrac{\partial^n\delta(w-z)}{\partial z^n},\quad a,b\in V,
\]
where $c_n\in V$,
$\delta(w-z)=\sum\limits_{m\in \mathbb Z} w^m z^{-m-1}$
is the formal delta-function.
The formal Fourier transformation
\begin{equation}\label{eq: ResLambdaProd}
[Y(a,z)_\lambda Y(b,z)] =
  \mathrm{Res}_{w=0} \exp\{\lambda (w-z)\}[Y(a,w),Y(b,z)]
\end{equation}
is called the $\lambda $-bracket on the space
of fields $\{Y(a,z)\mid a\in V\}$.
Here $\lambda $ is a new formal variable, and
 $\mathrm{Res}_{w=0}F(w,z) $ means the residue at
$w=0$, i.e., the formal series in $z$ that is
a coefficient of $F(w,z)$ at~$w^{-1}$.

The algebraic properties of the $\lambda $-bracket \eqref{eq: ResLambdaProd}
lead to the formal definition of a conformal algebra
over a field $\Bbbk $ of characteristic~0.

\begin{definition}[{\cite{Kac1998}}]   \label{defn:Conformal Algebra}
A conformal algebra is a left (unital) module $C$ over the
polynomial algebra $H=\Bbbk [T]$ endowed with a binary
$\Bbbk$-linear operation
\begin{equation}\label{eq:LambdaProduct}
  (\cdot{}_\lambda \cdot): C\otimes C \to C[\lambda ]  ,
\end{equation}
such that
$(Ta_\lambda b) = -\lambda (a_\lambda b)$,
$(a_\lambda D b) = (T+\lambda )(a_\lambda b)$.
\end{definition}

In terms of fields, $T$ is just the ordinary derivation with respect to~$z$.

Every conformal algebra can be represented by formal distributions
over an ordinary algebra. Let $C$ be an object described by
Definition \ref{defn:Conformal Algebra}.
Consider the space of Laurent polynomials $\Bbbk[t,t^{-1}]$ as a
right $H$-module with respect to the following action:
$f(t)T = -f'(t)$.
Then
\[
\mathcal A(C) = \Bbbk[t,t^{-1}]\otimes _H C
\]
carries the natural algebra structure:
\[
  (f\otimes _H a)\cdot (g\otimes_H b) = (g\otimes_H 1)(f\otimes _H (a_{-T} b)),
  \quad a,b\in C,\ f,g\in \Bbbk[t,t^{-1}].
\]
The space of formal distributions $\mathcal A(C)[[z,z^{-1}]]$
that consists of all series
\[
Y(a,z)=\sum\limits_{n\in \mathbb Z} (t^n\otimes_H a)z^{-n-1},
\quad
a\in C,
\]
can be endowed with the action of $T=d/dz$ and with a $\lambda $-bracket
$(\cdot {}_\lambda \cdot)$
similar to \eqref{eq: ResLambdaProd}, where the commutator
is replaced with the ordinary product of distributions.
Then
\[
  (Y(a,z)_\lambda Y(b,z)) =Y((a_\lambda b),z), \quad a,b\in C,
\]
i.e., $C$ is isomorphic to a formal distribution conformal algebra
over $\mathcal A(C)$.

The algebra $\mathcal A(C)$ is called the coefficient algebra \cite{Roitman1999}
of $C$, or annihilation algebra \cite{Kac1998}.

\begin{definition}[{\cite{Roitman1999}}]\label{defn:VarConformal}
Let $\mathcal V$ be a variety of algebras (associative, alternative, Lie, etc.).
Then a conformal algebra $C$ is said to be $\mathcal V$-conformal algebra
if $\mathcal A(C)$ belongs to~$\mathcal V$.
\end{definition}

Associative and Lie conformal algebras, their representations, and cohomologies
have been studied in a series of papers, e.g.,
\cite{BKV1999, DK1998, CK1997, Roitman2000,
BKL2003, Retakh2001, KacRetakh2008}.
In particular, associative conformal algebras naturally appear
in the study of representations of Lie conformal algebras.

\begin{example}\label{exm:CurrentAlgebra}
Consider one of the simplest (though important) examples of conformal algebras.
Suppose $A$ is an ordinary algebra
(not necessarily associative or Lie). Then the free $H$-module
\[
  \mathrm{Cur}\,A = H\otimes A
\]
generated by the space $A$ endowed
with the $\lambda $-bracket
$(f(T)\otimes a)_\lambda (g(T)\otimes b) =
  f(-\lambda )g(T+\lambda )\otimes ab$,
is called the current conformal algebra.
\end{example}

If $A$ belongs to a variety $\mathcal V$ defined by a family of
polylinear identities then $\mathrm{Cur}\, A$ is a $\mathcal V$-conformal
algebra.

Certainly, current conformal algebras and their subalgebras do not exhaust
the entire class of conformal algebras. For example, $W=\Bbbk[T,x]$
with respect to the operation
\[
(f(T,x)_\lambda g(T,x)) = f(-\lambda , T)g(T+\lambda , x+\lambda)
\]
is an associative conformal algebra (called Weyl conformal algebra \cite{Roitman1999}),
and $\mathrm{Vir}=\Bbbk[T]$ with respect to
\[
 ( f(T)_\lambda g(T) ) = f(-\lambda )g(T+\lambda )(T+2\lambda )
\]
is a Lie conformal algebra (called Virasoro conformal algebra \cite{Kac1998}).

Conformal algebra is said to be finite if it is a finitely generated $H$-module.

\section{Dialgebras}
The following notion appears naturally from a
certain noncommutative analogue of Lie homology theory.

\begin{definition}[{\cite{Loday1993}}]\label{defn:LeibnizAlgebra}
A (left) Leibniz algebra is a linear space $L$ with a bilinear operation
$[\cdot ,\cdot ]$ such that
\[
  [x,[y,z]]= [[x,y],z] + [y,[x,z]], \quad x,y,z\in L.
\]
\end{definition}

The defining identity means that the operator of
left multiplication $[x,\cdot]$ is a derivation of $L$.
Leibniz algebras are the most popular noncommutative generalizations
of Lie algebras. The following
structures play the role of associative enveloping algebras for
Leibniz algebras.

\begin{definition}[{\cite{Loday1995}}]\label{defn:AssocDialgebras}
An associative dialgebra (or {\em diassociative algebra}) is a linear space $D$ endowed
with two bilinear operations $(\cdot\dashv\cdot)$, $(\cdot\vdash\cdot)$
such that
\begin{gather}
x \dashv (y\vdash z) = x \dashv (y \dashv z), \quad
(x \dashv y)\vdash z = (x \vdash y) \vdash z,
                                                    \label{eq:0-ident}\\
x \vdash (y\vdash z) = (x \vdash y) \vdash z, \label{eq:DiAss1}\\
x \dashv (y\dashv z) = (x \dashv y) \dashv z,\label{eq:DiAss2}\\
x \vdash (y\dashv z) = (x \vdash y) \dashv z, \label{eq:DiAss3}
\end{gather}
for all $x,y,z\in D$.
\end{definition}

In particular, the operation
$[a,b]=a\vdash b - b\dashv a$, $a,b\in D$,
turns a diassociative algebra $D$ into a Leibniz algebra
denoted by $D^{(-)}$.

A systematical study of diassociative algebras was performed in
\cite{Loday2001}.
Also, in \cite{LiuDong2005} and \cite{Chap2001}
the notions of alternative and commutative dialgebras were introduced.
These definitions also appear in the general categorical approach
using the language of operads \cite{Kol2008a}.

Shortly speaking, an operad $A$
 is a collection of spaces $A(n)$, $n\ge 1$, such that a
 composition rule $A(n)\otimes A(m_1)\otimes \dots \otimes A(m_n)
 \to A(m_1+\dots + m_n)$
and an action of a symmetric group are defined in such a way
that some natural axioms hold (associativity of a composition,
existence of a unit in $A(1)$, and equivariance of the composition
with respect to the symmetric group action).

A linear space $A$ over a field $\Bbbk $ may be considered
as an operad (see, e.g., \cite{Leinster2004} as a general reference),
where $A(n)=\mathrm{Hom}\,(A^{\otimes n}, A)$.
In the free operad denoted by $\mathrm{Alg}$,
the spaces $\mathrm{Alg}(n)$ are spanned by (planar)
binary trees with $n$ leaves.

An algebra structure on a linear space
$A$ is just a functor of operads $\mathrm{Alg} \to A$.
If $\mathcal V$ is a variety of algebras defined by
polylinear identities then there exists a free $\mathcal V$-operad
$\mathcal V\hbox{-}\mathrm{Alg}$
built on polylinear polynomials of the free $\mathcal V$-algebra.
There exists a canonical functor $\mathrm{Alg}\to \mathcal V\hbox{-}\mathrm{Alg}$,
and it is clear that an algebra $A$ belongs to $\mathcal V$
if and only if there exists a functor
$\mathcal V\hbox{-}\mathrm{Alg} \to A$ such that the following diagram is commutative:
\[
\begin{array}{c@{} c@{} c@{} c@{} c}
\mathrm{Alg} &          &\longrightarrow&         &A\\
             & \searrow && \nearrow \\
             &          &\mathcal V\hbox{-}\mathrm{Alg}
\end{array}
\]


A similar definition works for dialgebras. An operad
$\mathrm{Dialg}$ whose spaces are spanned by planar
binary trees with 2-colored vertices (colors 1~and~2 stand for $\vdash$
and $\dashv $, respectively) has an image equivalent to
the Hadamard product $\mathrm{Alg}\otimes \mathcal E$,
where $\mathcal E $ is the free $\mathcal V_c$-operad
corresponding to the variety $\mathcal V_c$ of associative and commutative
dialgebras (Perm-algebras), $\dim \mathcal E(n)=n$ (see \cite{Chap2001, Kol2008a} for details).

Suppose
$\mathcal V$ is a variety of algebras  defined by
polylinear  identities.
For a linear space $D$,  a functor $\mathrm{Dialg} \to D$
defines two bilinear operations $\vdash$ and~$\dashv$ on $D$. Conversely, any system
$(D,\vdash, \dashv)$ may be considered as a functor $\mathrm{Dialg} \to D$.

\begin{definition}[{\cite{Kol2008a}}]\label{defn:Var-Dialg}
A linear space $D$ with two bilinear operations $\vdash$ and~$\dashv$
is said to be di-$\mathcal V$-algebra if there exists a functor
$\mathcal{V}\hbox{-}\mathrm{Alg}\otimes \mathcal E\to D$
such that the  following diagram
is commutative:
\[
\begin{CD}
\mathrm{Dialg} @>>> D \\
@VVV @VVV \\
\mathrm{Alg}\otimes \mathcal E @>>> \mathcal V\hbox{-}\mathrm{Alg}\otimes \mathcal E
\end{CD}
\]
\end{definition}

We will mainly use the term ``di-$\mathcal V$-algebra'', but this is the same as 
``$\mathcal V$-dialgebra''.

The last definition is easy to translate into the language of
identities. First, identify
$\mathrm{Alg}(n)$
with the space of polylinear non-associative polynomials in
$x_1,\dots, x_n$; for  $\mathrm{Dialg}$ we have a similar interpretation.
Next, consider the following linear maps
$\Psi_k: \mathrm{Alg}(n)\to \mathrm{Dialg}(n)$, $k=1,\dots, n$:
\[
  \Psi_k: (x_{j_1}\dots x_k \dots x_{j_n}) \mapsto (x_{j_1}\vdash \dots \vdash
  x_k
   \dashv \dots \dashv x_{j_n}),
\]
assuming the bracketing $(\dots )$ on monomials is preserved.
Then we have

\begin{theorem}[{\cite{Kol2008a}}]\label{thm:Dias-Identity}
Assume $\{f_i\mid i\in I\}$ is the family of polylinear defining
identities of a variety $\mathcal V$.
 Then $D$ is a di-$\mathcal V$-algebra if and
 only if $D$ satisfies the identities $\Psi_k(f_i)=0$
 for all $i\in I$, $k=1, \dots, \deg f_i$.
\end{theorem}

If $\mathcal V$ is the variety of Lie algebras then
$f = x_1x_2+x_2x_1$ is one of its defining identities.
Since $\Psi_1(f)=x_1\dashv x_2 + x_2\vdash x_1$,
we can describe Lie dialgebras in terms of single operation, say,
$[a,b]=a\vdash b$. Then the class of Lie dialgebras coincides
with the class of Leibniz algebras.

Note that all di-$\mathcal V$-algebras satisfy the  relations
\eqref{eq:0-ident}, called 0-identities \cite{Kol2008a}.
The following approach to the definition of varieties
of dialgebras was proposed in~\cite{Pozh2009}.

Let $D$ be a dialgebra that satisfies 0-identities. Then
the linear span $D_0$ of all elements
$a\vdash b - a\dashv b$, $a,b\in D$, is an ideal of~$D$.
The quotient $\bar D=D/D_0$ is an ordinary algebra with a
single operation. Moreover, the following actions are well-defined:
\[
  \begin{aligned}
   & \bar D\otimes D \to D, \\
   & (a+D_0)\otimes b \mapsto a\vdash b,
  \end{aligned}
  \qquad  \qquad
  \begin{aligned}
   & D\otimes \bar D \to D, \\
   & a\otimes (b+D_0) \mapsto a\dashv b.
  \end{aligned}
\]
Denote by $\hat D$ the split null extension $\bar D\oplus D$,
assuming $D^2=0$.


\begin{theorem}[{\cite{Pozh2009}}]\label{thm:EilengerghDefn}
Suppose $\mathcal V$ is a variety of algebras with
polylinear defining identities. Then $D$ is a
di-$\mathcal V$-algebra if and only if $D$ satisfies
the 0-identities and $\hat D$ is an algebra from~$\mathcal V$.
\end{theorem}

A curious relation between conformal algebras and dialgebras
was noted in~\cite{Kol2008a}. It turns out that
if $C$ is a $\mathcal V$-conformal algebra in the sense
of Definition \ref{defn:VarConformal} then
the same linear space endowed with just two operations
\[
 a\vdash b = (a_\lambda b)|_{\lambda =0}, \quad
 a\dashv b = (a_\lambda b)|_{\lambda =-T}, \quad a,b\in C,
\]
is a di-$\mathcal V$-algebra denoted by $C^{(0)}$.
Conversely, every di-$\mathcal V$-algebra can be embedded into
an appropriate $\mathcal V$-conformal algebra.
The last statement easily follows from Theorem \ref{thm:EilengerghDefn}
and

\begin{theorem}[{c.f. \cite{KolGub2009}}]\label{thm:CurrEmbedding}
Let $D$ be a dialgebra satisfying the 0-identities.
Then the map $D\to H\otimes \hat D$, $a\mapsto 1\otimes (a+D_0) + T\otimes a$,
$a\in D$, is an injective homomorphism of
dialgebras $D\to (\mathrm{Cur}\,\hat D)^{(0)}$.
Therefore, $D$ is a di-$\mathcal V$-algebra if and only if there exists
a $\mathcal V$-algebra $A$ such that $D\subseteq (\mathrm{Cur}\,A)^{(0)}$.
\end{theorem}

Thus, there are three equivalent definitions of what is a dialgebra of
a given variety provided by Theorems~\ref{thm:Dias-Identity},
\ref{thm:EilengerghDefn}, and~\ref{thm:CurrEmbedding}.

\section{Some Classical Theorems for Leibniz Algebras}

Since Leibniz algebras are just Lie dialgebras in the sense of
Definition~\ref{defn:Var-Dialg}, we may use
Theorem~\ref{thm:CurrEmbedding} to get natural generalizations of
 some classical statements on Lie algebras to the class of Leibniz algebras.
These are: the Engel Theorem, the Poincar\'e---Birkhoff---Witt (PBW) Theorem,
and the Ado Theorem.

We will need the following statement (c.f.~Theorem~3 in~\cite{Kol2008b}).

\begin{theorem}\label{thm:ConfRepresentation}
Let $L$ be a Leibniz algebra, and let $V$ be a module over the
Lie algebra $\bar L$. Then there exists an injective homomorphism
$\rho: L\to (\mathrm{Cur}\,\mathrm{gl}(V\oplus (L\otimes V)))^{(0)}$
of Leibniz algebras.
\end{theorem}

\begin{proof}
For every $x\in L$, denote $\bar x=x+L_0\in \bar L$ and define
$\rho(x)\in H\otimes \mathrm{gl}(V\oplus (L\otimes V))$ as follows:
\[
\rho(x)  = 1\otimes \rho_0(x) + T\otimes \rho_1(x),
\quad \rho_i(x)\in \mathrm{gl}(V\oplus (L\otimes V)),
\]
where
\[
\begin{aligned}
\rho_0(x) &: v\mapsto \bar x v, \\
\rho_0(x) &: a\otimes v \mapsto a\otimes \bar x v + [x,a]\otimes v, \\
\rho_1(x) &: v\mapsto x\otimes v, \\
\rho_1(x) &: a\otimes v \mapsto 0
\end{aligned}
\]
for all $a\in L$, $v\in V$.
It is clear that $\rho $ is injective ($\rho_1(x)\ne 0$ for $x\ne 0$).
Let us check that $\rho $ is a homomorphism of Leibniz algebras.
First,
$(\rho(x)_\lambda \rho(y))
= [1\otimes \rho_0(x)-1\otimes \lambda \rho_1(x), 1\otimes \rho_0(y)+(T+\lambda)\rho_1(y)]$.
for all $x,y\in L$.
Hence,
$[\rho(x), \rho(y)]=1\otimes [\rho_0(x), \rho_0(y)] +
 T\otimes [\rho_0(x),\rho_1(y)]$.
Next, it is straightforward to compute
\[
\begin{aligned}{}
[\rho_0(x),\rho_0(y)] &:v+(a\otimes w)\mapsto
   [\bar x,\bar y]v + a\otimes [\bar x,\bar y]w + [[x,y],a]\otimes w, \\
[\rho_0(x),\rho_1(y)] &:v+(a\otimes w)\mapsto [x,y]\otimes v
\end{aligned}
\]
for all $v,w\in V$, $a\in L$.
Therefore,
$[\rho_0(x),\rho_0(y)]=\rho_0([x,y])$,
$[\rho_0(x),\rho_1(y)]=\rho_1([x,y])$, i.e.,
$[\rho(x), \rho(y)]=\rho([x,y])$.
\end{proof}

The following statement immediately follows from
Theorem~\ref{thm:ConfRepresentation} applied to $V=L$.

\begin{theorem}[{\cite{AyOmir1998, Pats2007, Barn2010}}]\label{thm:Engel}
Let $L$ be a finite-dimensional Leibniz algebra
such that all operators $[x,\cdot]\in \mathrm{End}\, L$ are nilpotent.
Then $L$ itself is a nilpotent Leibniz algebra.
\end{theorem}

Recall that for a Lie algebra $L$ the classical PBW Theorem states
that the universal enveloping associative algebra $U(L)$ is isomorphic
(as a linear space) to the symmetric algebra $S(L)$.
For Leibniz algebras, the role of associative envelopes belongs to
diassociative algebras.

\begin{theorem}[{\cite{Loday2001, AymonGr2003}}]\label{thm:PBWLeibniz}
The universal enveloping diassociative algebra $Ud(L)$
of a Leibniz algebra $L$ is isomorphic (as a linear space) to
$U(\bar L)\otimes L$.
\end{theorem}

As in the case of Lie algebras, the main technical difficulty in the
proof of the PBW Theorem for Leibniz algebras is to show that ``normal"
monomials are linearly independent. In \cite{BokutChen2009}, another
proof of this independence was obtained by making use of Gr\"obner---Shirshov
bases theory for diassociative algebras. However, one may just apply
Theorem~\ref{thm:ConfRepresentation} to $V=U(\bar L)$,
see \cite{Kol2008b} for details.

Another interesting question is similar to the Ado Theorem: Whether
a finite-dimensional Leibniz algebra can be embedded into a
finite-dimensional diassociative algebra?
It turns out, the answer is positive. Indeed, it is enough to apply
Theorem~\ref{thm:ConfRepresentation} to $V=\Bbbk $, a trivial 1-dimensional
module over $\bar L$. In particular, we may conclude that an $n$-dimensional
Leibniz algebra can be embedded into a diassociative algebra $D$ such
that $\dim D\le 2(n+1)^2$.

\section{Di-Jordan algebras}

A diassociative algebra $D$ turns into a Leibniz algebra $D^{(-)}$
if we define the bracket $[x,y]=x\vdash y - y\dashv x$.
This is natural to expect that if we define new operation
\[
 x\circ y =  x\vdash y + y\dashv x, \quad x,y\in D,
\]
then the algebra $D^{(+)}=(D,\circ ) $
 obtained would be a noncommutative analogue
of a Jordan algebra. Roughly speaking, it relates to Jordan algebras in the
same way as Leibniz algebras relate to Lie algebras.

This is indeed a di-Jordan algebra; the commutativity identity turns
into 
\[
\Psi_1(x_1x_2-x_2x_1)=x_1\dashv x_2 - x_2\vdash x_1,
\]
so we may
describe this algebra with only one operation.
Objects of this type appeared also in~\cite{VF2009, Brem2009}.

\begin{definition}
A di-Jordan algebra is a linear space with a bilinear product satisfying the
following identities:
\begin{equation}   \label{eq:JorDiasAlg}
\begin{gathered}[]
[x_1, x_2]x_3= 0, \\
(x_1^2,x_2,x_3)=2(x_1,x_2,x_1x_3),
\quad
x_1(x_1^2 x_2)=x_1^2(x_1 x_2).
\end{gathered}
\end{equation}
\end{definition}

Here $[a,b]$ and $(a,b,c)$ stand for the commutator $ab-ba$ and associator $(ab)c-a(bc)$,
respectively. The first identity in \eqref{eq:JorDiasAlg} comes from the 0-identities,
the second and third appear from the Jordan identity.
In \cite{BremPeresi2010} these algebras were called
semi-special quasi-Jordan algebras.

Recall that a Jordan algebra $J$ is said to be special if there exists
an associative algebra $A$ such that $J\subseteq A^{(+)}$. The class of
all homomorphic images of all special Jordan algebras is a variety denoted
by $\mathrm{SJ}$.
This is well-known that $\mathrm{SJ}$ does not coincide with the variety
of all Jordan algebras.
Those defining identities of $\mathrm{SJ}$ that do
not hold in all Jordan algebras
are called special identities (or s-identities, for short).

I was shown in \cite{Glenie1970} that the minimal degree of an
s-identity is equal to~8. However, the description of all s-identities
is still an open problem.

For di-Jordan algebras, the same theory makes sense.

\begin{definition}[{\cite{BremPeresi2010}}]
A di-Jordan algebra $J$ is said to be special if
there exists a diassociative algebra $D$ such that
$J\subseteq D^{(+)}$.
\end{definition}

It is clear that the class of all homomorphic images of all
special di-Jordan algebras is a variety. Let us denote this variety
by $\mathrm{DiSJ}$. The notion of an s-identity for di-Jordan algebras
is a natural generalization of s-identities for Jordan algebras.

The following statement was proved in  \cite{BremPeresi2010}
by making use of computer algebra methods.

\begin{theorem}[{\cite{BremPeresi2010}}]\label{thm:Di-Sidentities}
{\em 1.} For di-Jordan algebras, there are no s-identities of degree $\le 7$;\\
{\em 2.} There exists an identity of degree~8 that holds on all special di-Jordan
algebras and on all Jordan algebras, but does not hold on all
di-Jordan algebras.
\end{theorem}

On the other hand, the variety $\mathrm{SJ}$ leads to the notion of
a di-$\mathrm{SJ}$-algebra by Definition~\ref{defn:Var-Dialg}. It turns
out that these two different approaches lead to the same class of
dialgebras.

\begin{theorem}[{\cite{Voron2010}}]\label{thm:DiSJ=SJDi}
The variety of di-$\mathrm{SJ}$-algebras coincides with $\mathrm{DiSJ}$.
\end{theorem}

This fact allows to deduce a correspondence between s-identities for
Jordan algebras and dialgebras.

\begin{theorem}[{\cite{Voron2010}}]\label{thm:SId-Correspond}
Let $f(x_1,\dots, x_n)$ be a polylinear s-identity for Jordan algebras. Then
$\Psi_k f$, $k=1,\dots, n$,
is an s-identity for di-Jordan algebras.
Conversely, if $g(x_1,\dots, x_n)$ is an
s-identity for di-Jordan algebras then
\[
  g(x_1,\dots, x_n)=\sum\limits_{k=1}^n g_k, \quad g_k=\Psi_k(f_k)
\]
for some nonassociative polynomials $f_k(x_1, \dots, x_n)$,
and at least one of $f_k$ is an s-identity for Jordan algebras.
\end{theorem}

Note that Theorem~\ref{thm:SId-Correspond} works for polylinear identities
only, so it says nothing about the identity
from Theorem~\ref{thm:Di-Sidentities}(2).

A series of classical results for special Jordan algebras can be transferred
to dialgebras. In particular, the Shirshov---Cohn Theorem states that
every 2-generated Jordan algebra is special. It turns out that the free
2-generated di-Jordan algebra is special, but its homomorphic image
may not be special \cite{Voron2010}. However, Theorem~\ref{thm:CurrEmbedding}
implies that every 1-generated di-Jordan algebra is special.

Another problem on di-Jordan algebras concerns their relation to Leibniz algebras.
The classical Tits---Kantor---Koecher construction allows to build a
Lie algebra $T(J)$ for a given Jordan algebra $J$ in such a way that structure
of $J$ is closely related with the structure of $T(J)$.
This is natural to expect \cite{VF2009} that a similar construction for a di-Jordan algebra
should lead to Lie dialgebra, i.e., Leibniz algebra.

Conformal algebras allow to solve this problem.

Let $J$ be a di-Jordan algebra, and let $\hat J$ stands
for the split null extension $\bar J\oplus J$
(see Theorem~\ref{thm:EilengerghDefn}).
This is a Jordan algebra, and it follows from Theorem~\ref{thm:CurrEmbedding}
that $J\subseteq (\mathrm{Cur}\,\hat J)^{(0)}$.

Denote by
\[
T(\hat J) = {\hat J}^+ \oplus S(\hat J) \oplus {\hat J}^-
\]
the Tits---Kantor---Koecher construction \cite{Kantor1964,Koecher1967,Tits1962}
for~$\hat J$.
Here ${\hat J}^{\pm}$ are linear spaces isomorphic to $\hat J$, and
$S(\hat J)\subseteq \mathrm{End}\, \hat J\oplus  \mathrm{Der}\, \hat J$
is spanned by
\[
  U_{a,b} = L_{ab} + [L_a,L_b], \quad a,b\in \hat J,
\]
where $L_x$ denotes the operator of left multiplication:
$L_x(y)=xy$.
The images of $a\in J$ in the isomorphic copies $J^{\pm}$ are denoted
by~$a^\pm$.
This is a Lie algebra, $J^+$ and $J^-$ are its abelian subalgebras, and
$[a^-, b^+] = L_{ab} + [L_a, L_b]$ for $a,b\in J$. Therefore,
$\mathcal L(J) = (\mathrm{Cur}\,T(\hat J))^{(0)}$ is a
Leibniz algebra.
Then the elements $1\otimes (a+J_0)^\pm + T\otimes a^\pm \in \mathcal L(J)$,
$a\in J$,
generate a Leibniz algebra $T(J)$

\begin{theorem}[{\cite{KolGub2009}}]
Let $J$ be a di-Jordan algebra.
Then $T(J)$ is a solvable Leibniz algebra if and only if
$J$ is a Penico solvable \cite{JacJordan}; $T(J)$ is nilpotent if
and only if so is~$J$.
\end{theorem}

\section{On Embeddings of Conformal Algebras}

One of the basic facts about associative algebras states that
every finite-dimen\-sional associative algebra $A$ can be presented
by matrices. Indeed, even if $A$ does not contain a unit element,
we may consider $A^\#=A\oplus \Bbbk 1$, and then there is a
faithful representation $L:A\to \mathrm{End}\,A^\#$,
$L(a): x\to ax$.

For a conformal algebra $C$ of rank $n$ over $H$, the role of $\mathrm{End}\, A$ belongs
to $\mathrm{Cend}\,C$, which is isomorphic to the conformal algebra
of $n\times n$ matrices over the Weyl conformal algebra.
The following properties define an analogue of the unit element for
conformal algebras.

\begin{definition}[{\cite{Retakh2001}}]\label{defn:Unit}
Suppose $C$ is a conformal algebra. An element $e\in C$
is said to be a (conformal) unit in $C$ if
$(e_\lambda x)|_{\lambda=0} = x$ for all $x\in C$ and
$e_\lambda e = e$.
\end{definition}

Associative conformal algebra with a unit has a very natural structure.

\begin{proposition}[{\cite{Retakh2001}}]\label{prop:DiffAlgebra}
Let $C$ be a semisimple associative conformal algebra with a unit.
Then there exists an associative algebra $A$ with a locally
nilpotent derivation $\partial $ such that $C\simeq H\otimes A$
with respect to the operation
\[
  (f(T)\otimes a)_\lambda (g(T)\otimes b) = f(-\lambda )g(T+\lambda )\otimes
  a\exp\{\lambda \partial\}b,\quad a,b\in A,\ f,g\in H.
\]
\end{proposition}

\begin{remark}
In particular, if $\partial =0$ then such $C$ is just the current algebra
$\mathrm{Cur}\,A$; if $A=\Bbbk[x]$ and $\partial = d/dx$ then $C$ is the
Weyl conformal algebra.
\end{remark}

This is the reason why the following problem \cite{Retakh2004} makes sense:
Is it possible to join a conformal unit to an associative conformal algebra.
Moreover, a finite associative conformal algebra can be embedded into
matrices over the conformal Weyl algebra if and only if one can join a
unit to this conformal algebra.
The following statement answers positively to this question.

\begin{theorem}[{\cite{Kol2010}}]\label{thm:UnitEmbedding}
If $C$ is a finite associative conformal algebra which is a
torsion-free $H$-module then there exists an associative
conformal algebra $C_e$ with a unit such that $C\subseteq C_e$.
\end{theorem}

The last Theorem does not hold for all conformal algebras.
For example, consider the free $H$-module $C$ generated by the space
$W=\Bbbk[x]\oplus \Bbbk w$ with the following operation on generators:
\[
  W_\lambda w = 0, \quad 
 f(x)_\lambda g(x)=f(x-T-\lambda)g(x ),  \quad 
 w_\lambda f(x) = f(T)w,
\]
for $f,g\in \Bbbk[x]$. This is an associative conformal algebra.
Assume there exists a unital conformal algebra $C_e$
with unit $e$
such that $C\subseteq C_e$. Then associativity implies
\[
 ((e_\lambda w)_0 x^n) = (e_\lambda (w_{-\lambda} x^n)) = e_\lambda T^nw =
   (T+\lambda )^n (e_\lambda w), \ n\ge 1. 
\]
This is impossible since $(e_\lambda w)$ is a polynomial in $\lambda $,
and its degree does not depend on the choice of $n$, except for
$(e_\lambda w)=0$. But $(e_{0 } w)=w$ by the definition of a conformal 
unit. The contradiction obtained shows that $C_e$ does not exist.

Another problem on embeddings of conformal algebras concerns the following 
observation.
If $C$ is an associative conformal algebra
with operations $T:C\to C$ and $(\cdot{}_\lambda \cdot):C\otimes C\to C[\lambda ]$
then the same module over $H=\Bbbk[T]$ with respect to the new operation
\[
[x_\lambda y ]  = (x_\lambda y) - (y_{\mu }x)|_{\mu=-T-\lambda}
\]
is a Lie conformal algebra \cite{Kac1998}. It is  natural
to denote this conformal algebra by~$C^{(-)}$.

In contrast to the case of ordinary algebras, there exist Lie conformal algebras
that can not be embedded into associative conformal algebras \cite{Roitman2000}.
However, it is unknown whether the following statement is true.

\begin{conjecture}\label{conj:Ado}
Suppose $L$ is a finite Lie conformal algebra which is a torsion-free
$H$-module.
Then $L$ can be embedded into an associative conformal algebra with unit.
\end{conjecture}

It was shown in \cite{DK1998}
that every such $L$ has a maximal solvable ideal $R$, so $L/R$ is semisimple.

The conjecture obviously holds for semisimple conformal algebras.
In \cite{Roitman2006}, it was shown that a nilpotent Lie conformal algebra
can be embedded into a nilpotent associative conformal algebra with the same index
of nilpotency.
This proves Conjecture~\ref{conj:Ado} for nilpotent algebras, but it 
actually holds in much more general class of Lie conformal algebras that 
includes finite torsion-free solvable algebras.

The idea comes from the construction of a
conformal representation for a Leibniz algebra in 
Theorem~\ref{thm:ConfRepresentation}.

Let us first consider a Lie conformal algebra of the type
$L=\mathrm{Cur}\,\mathfrak g$, where $\mathfrak g$ is a Lie algebra, 
$\dim \mathfrak g<\infty $. 
This is straightforward to check that the embedding built in the proof of 
Theorem~\ref{thm:ConfRepresentation} 
is in fact a homomorphism of conformal algebras. 
This proves Conjecture~\ref{conj:Ado}
without a reference to the classical Ado Theorem. 

 In the more general case,  the Lie Theorem for conformal algebras
proved in \cite{DK1998} allows to deduce the following fact.

\begin{theorem}[{\cite{Kol2010}}]
Suppose $L$ is a Lie conformal algebra which is a semi-direct product 
of a current conformal algebra $\mathrm{Cur}\,\mathfrak g$
($\dim\mathfrak g<\infty$)
and a finite torsion-free solvable Lie conformal algebra $R$.
Then there exists a finite-dimensional associative algebra $A$
such that $L\subseteq (\mathrm{Cur}\,A)^{(-)}$.
\end{theorem}

\section*{Acknowledgments}
This work was partially supported by 
RFBR 09-01-00157, SSc 3669.2010.1,
SB RAS Integration project N 97, 
Federal Target Grants  02.740.11.0429,  02.740.11.5191, 14.740.11.0346,
and by the ADTP Grant 2.1.1.10726. 

\bibliographystyle{plain}
\bibliography{kolesnikov}

\end{document}